\documentclass{amsart}
\usepackage[all]{xy}
\usepackage{amsmath}
\usepackage{amsfonts}
\usepackage{amssymb}
\usepackage[pdftex]{graphicx}
\usepackage{xy}
\usepackage{latexsym}
\usepackage{epsfig}
\usepackage{epsfig}
\usepackage{mdwtab}
\usepackage{float}

\date{}
\newtheorem{theo}{Theorem}
\newtheorem{lem}[theo]{Lemma}
\newtheorem{prop}[theo]{Proposition}

\newtheorem{defi}{Definition}

\newtheorem{rem}{Remark}

\newcommand{\beq}{\begin{equation}}
\newcommand{\eeq}{\end{equation}}

\newcommand{\h}{{\mathbb{H}}}
\newcommand{\Z}{{\mathbb{Z}}}
\newcommand{\R}{{\mathbb{R}}}

\newcommand{\tq}{\ |\ }

\newcommand{\type}{\LARGE{ \in}}

\floatname{figure}{\bf Figure}
\newfloat{table}{thp}{Table}[section]
\floatname{table}{\bf Table}

\usepackage{epsfig}

\usepackage{multirow}

\usepackage{array}

\begin{document}

\title[The Borsuk-Ulam theorem for 
the Seifert manifolds]{The Borsuk-Ulam theorem for 
the Seifert manifolds having   Flat geometry} 

\footnote{This work is part of the Projeto tem\'atico Topologia Alg\'ebrica e
Geom\'etrica  FAPESP 
2016/24707-4 }

\author{A. Bauval, D. L.\ Gon\c calves and  C. Hayat } 


\begin{abstract}

\noindent

\emph{Let $M$ be a Seifert manifold  which belongs to  the geometry     Flat. In this 
work we determine   all the free involutions $\tau$ on $M$, and the  Borsuk-Ulam indice of
$(M,\tau)$. } 

\end{abstract}

\maketitle

\section{Introduction}   Given a  pair  $(M, \tau)$, where $\tau$ is a free involution on the space $M$, 
  the following generalization of the question raised by Ulam, has been studied:   {\it Given 
   $(M, \tau)$ determine all integers  $n>0$ such that  every map $f:M \to \R^n$ has the property that 
   there  is an $x\in M$ such that $f(x)=f(\tau(x))$}. When $n$  belongs   to this family we then  say 
   that  {\it     $(M, \tau)$ has the Borsuk-Ulam property with respect to  maps into $\R^n$, or  $(M,\tau)$ verifies the  Borsuk-Ulam theorem for $\R^n$.}  The greatest integer $n$ such that \ $(M,\tau)$ verifies the  Borsuk-Ulam theorem for $\R^n$ is called the $\Z_2$-index of $((M,\tau);M/\tau)$.
  
 In the introductions of  \cite{adcp2} an \cite{ghz}, 
 the problem of determine the values $n$ for which  $(M,\tau)$ verifies the  Borsuk-Ulam theorem for $\R^n$  is explained in more details, and  some historical notes about this problem    including the information about recent results are provided.  The work \cite{adcp2} study the above   problem  
  for the 
   families of Seifert  3-manifolds having  geometries  
either  $S^2\times \R$ or   $S^3$.  This work is a natural continuation of \cite{adcp2} where here we consider  the family of  the 
Flat 3-manifolds. These manifolds  have the 3-torus as a finite covering. They
are classified as the ten Flat Riemann manifolds \cite{w}. 
They are described as Seifert manifolds in \cite{orlik},  subsection 8.2  where six are
orientable  and four non-orientable.

Our method consists in determining  the pairs $(M, \tau)$, where $M$ runs over the double coverings of a given manifold $N$, and it  makes use of the group theoretic Reidemeister-Schreier method, 
\cite{mks} Chapitre 2 Section 2.3, \cite{ZVC} page 24. 

 The involution $\tau$  is then the involution  associated to the double covering.   The answer  to the  Borsuk-Ulam property is obtained  by cohomological properties  of the  cohomology class of $N$ determined by   the epimorphism of fundamental groups deduced from the involution.

Applying our method to  the ten Flat manifolds we obtain  the  finite number of involutions  on each of them and we built 
the graph, see FIGURE \ref{fig:flat} which incorporates the projection associated to the double covering and the associated  $\Z_2$-index. The  statement of the main result, Theorem \ref{them:flat}, is obtained reading this graph.

\section{Preliminaries}

Let us recall some terminology and known results useful in this paper.

\begin{defi} \label{defi:equiv0} We say that $(M_1,\tau_1)$ and $(M_2,\tau_2)$  are equivalent if
there is     a homeomorphism $h\colon M_1 \to M_2$ which is equivariant with respect to the $\Z_2$ free actions
provided by the involutions.
\end{defi}

In this article, when we say  that $M$ admits a {\it unique involution}, it means unique up to this equivalence.

The following lemma gives a criterion to decide when two double coverings are equivalent under certain hypothesis 
on the quotient manifolds.

\begin{lem} \label{lem:equi}   Let $(M_i, \tau_i)$ $i=1,2$,  be a pair such that $\tau_i $ is a free involution on  a 3-manifold $M_i$ and  
  $\varphi_i$ the associated epimorphism $\varphi_i\colon\pi_1(N_i)
\twoheadrightarrow\Z_2$. We have: \\
a) If $(M_1,\tau_1)$ is equivalent to $(M_2,\tau_2)$ then there exists an  isomorphism  $\theta:\pi_1(N_1)\to \pi_1(N_2)$ 
such that $\varphi_1=\varphi_2\circ\theta$.\\
b) If $N_i$ $i=1,2$  have the property that    any isomorphism  of the fundamental group of these manifolds is realizable as a homeomorphism of the manifolds,  then $(M_1,\tau_1)$ is equivalent to $(M_2,\tau_2)$  if there is  an isomorphism $\theta:\pi_1(N_1)\to \pi_1(N_2)$ such that  $\varphi_1=\varphi_2\circ\theta$.\\
\end{lem}  
\begin{proof}  a)  Let $p_i\colon M_i\to M_i/\tau_i=N_i$ be the projection of the double covering, if there is a homeomorphism $h\colon M_1 \to M_2$ such that $h\circ\tau_1=\tau_2\circ h$, then 
there exists a homeomorphism $\bar h\colon N_1\to N_2$ such that $\bar h\circ p_1=p_2\circ h.$\\
Let ${\bar x}\in \pi_1(N_1)$  such that $\varphi_1({\bar x})=0$. Hence there exists $x\in \pi_1(M_1)$ such that ${\bar x}=p_{1 \sharp}(x)$ and $\varphi_2\circ \bar h_\sharp ({\bar x})=\varphi_2\circ \bar h_\sharp \circ p_{1 \sharp}(x)=\varphi_2\circ  p_{2 \sharp}\circ h_\sharp(x)=0.$
\\
Let ${\bar x}\in \pi_1(N_1)$  such that $\varphi_1({\bar x})=1$ and let us suppose that $\varphi_2\circ \bar h_\sharp ({\bar x})=0$.  Hence there exists $y\in \pi_1(M_2)$ such that $\bar h_\sharp({\bar x})= p_{2 \sharp}(y)$. By hypothesis $h_\sharp $ and $\bar h_\sharp$ are isomorphisms, hence there exists $u\in \pi_1(M_1)$ such that $y=h_\sharp(u)$ and from the equalities $\bar h_\sharp({\bar x})= p_{2 \sharp}\circ h_\sharp(u)=\bar h_\sharp \circ p_{1 \sharp}(u)$ it follows ${\bar x}=p_{1 \sharp}(u)$ and $\varphi_1({\bar x})=0$ which is a contradiction. Hence the conclusion $\varphi_1=\varphi_2\circ \bar h_\sharp$.\\

b) Let us suppose that there exists a homeomorphism $\bar h\colon N_1\to N_2$ such that $\theta=\bar h_\sharp$. Using, for example \cite{mas}, Theorem 5.1, page 156, we know that $\bar h\circ p_1$ admits a lifting $h$ through $p_2$ such that $\bar h\circ p_1= p_2\circ h$ if $(\bar h \circ p_1)_\sharp(\pi_1(M_1))\subset p_{2 \sharp}(\pi_1(M_2))$.\\
Let $x\in \pi_1(M_1)$, by hypothesis we have
 $\varphi_2\circ \bar h_\sharp(p_{1\sharp}(x))=\varphi_1(p_{1\sharp}(x))=0.$\\
 Hence there exist  $y\in \pi_1(M_2)$ such that $\bar h_\sharp(p_{1\sharp}(x))=p_{2 \sharp}(y)$. This proves that 
 $(\bar h\circ p_1)_\sharp(\pi_1(M_1))\subset p_{2 \sharp}(\pi_1(M_2))$. The other inclusion is  obtained by the same method.\\
 \end{proof}

From \cite{ghz},  Theorems (3.1) and (3.2),  we have:
 
\begin{theo} \label{thm:p2}  Let $(M, \tau)$ be a pair such that $\tau $ is a free involution on  a 3-manifold $M$  and  denote by  $\varphi $ the associated epimorphism $\varphi\colon\pi_1(N)
\twoheadrightarrow\Z_2$ where $N=M/\tau$. 
\noindent
(i) \   One has   $\Z_2$-index of $((M,\tau);N)$ equals 1 if and only if the 
homomorphism $\varphi\colon\pi_1(N)\twoheadrightarrow\Z_2$  factors through the 
projection $\Z\twoheadrightarrow\Z_2 $.  \\
\noindent
(ii) \  One has  $\Z_2$-index of $((M,\tau);N)$ equals 3 if and only if   the cup-cube $[\varphi]^3 \not= 0$ where $
[\varphi]$ is the class of $\varphi$ in $H^1(N;\Z_2)$.
\end{theo}

Now and in all the following,  $M$ is a Seifert manifold in the classique sens of Seifert \cite{s}. 
Following the notation of  Orlik \cite{orlik}
$M$ will be 
described by a list of Seifert invariants  $$\{b; (\type, g); (a_1,b_1),\ldots ,(a_n,b_n)\}$$
where $b$ is an integer,  $g$ is the genus of the base 
surface (the orbit space obtained by identifying each $S^1$  fibre of $M$ to a point),  for each $k$, the integers $a_k,b_k$ are coprime. The type  $\type\in \{o_1,o_2,n_1,n_2,n_3\}$  reflects the orientations of the base and the total space of the Seifert fibration of $M$.

 We shall   use the following presentation of the fundamental 
group  of $M=\{b; (\type, g); (a_1,b_1),\ldots ,(a_n,b_n)\}$:
\begin{eqnarray}
\pi_1(M)=\left<{\begin{matrix}s_1,\ldots , s_n\\v_1,\ldots  ,v_{g'}\\h\end{matrix}\left|
\begin{matrix}
 [s_k,h]\quad\text{and}\quad s_k^{a_k}h^{b_k},&1\le k\le n\\
v_jhv_j^{-1}h^{-\varepsilon_j},&1\le j \le g'\\
s_1\ldots s_nVh^{-b}&\end{matrix}\right.}\right>, \label{BB}
\end{eqnarray}
where the generators and $g', V$ are described below.

\begin{itemize}
\item The orientability of the base surface and its genus $g$ determine the number $g'$ 
of the generators $v_j$'s and the word $V$ in the last relator of $\pi_1(M)$ as follows:\\
- when the base surface is orientable, $g'=2g$ and  $V=[v_1,v_2]\ldots[v_{2g-1},v_{2g}]$;\\
- when the base surface is non-orientable, $g'=g$ and  $V=v_1^2\ldots v_g^2$.
\item The generator $h$ corresponds to the generic regular fibre.
\item The generators $s_k$ for $1\le k\le n$ correspond to exceptional fibres.
\end{itemize}

\begin{defi}\label{defi:cd}  Let $d$ denote the number of indices $j$ such that $a_j$ is even,
 and $$c=ba+\sum_{j=1}^n b_j(a/a_j),$$
where $a$ denotes the least commun  multiple of the $a_j$'s. In the case where  the total space is non-orientable 
we  let $b$ be $0$ or $1$. 
\end{defi}

 In \cite{adcp1} the following proposition is proved.

\begin{prop}\label{prop:cube}   Let $[\varphi]\in H^1(N;\Z_2)$, where $N$ is a Seifert manifold and $\hat\gamma$ is the generator of $H^3(N;\Z_2)$.
\begin{itemize}
\item  Case 1 : $d=0$ and $c$ even, $[\varphi]^{3}=$
	\begin{itemize}
	\item[1)] $\varphi(h)\frac c2\hat\gamma$ when $N$ is of type $o_1$,
	\item[2)] $\varphi(h)(\frac c2+\sum\varphi(v_i))\hat\gamma$ when $N$ is of type $o_2$ or $n_1$,
	\item[3)] $\varphi(h)(\frac c2+g)\hat\gamma$ when $N$ is of type $n_2$,
	\item[4)] $\varphi(h)(\frac c2+\varphi(v_1)+g-1)\hat\gamma$ when $N$ is of type $n_3$,
	\item[5)] $\varphi(h)(\frac c2+\varphi(v_1)+\varphi(v_2)+g-2)\hat\gamma$ when $N$ is of type $n_4$.
	\end{itemize}
\item Case 2 : $d=0$ and $c$ odd, $[\varphi]^{3}=0$.
\item  Case 3 : $d>0$,
$[\varphi]^{3}=(\sum_1^n\varphi(s_j)\frac{a_j}2)\hat\gamma$.
\end{itemize}
\end{prop}

\section{ The Borsuk-Ulam theorem for Flat  manifolds}
 We discuss here the manifolds which have the 3-torus as a finite covering. They
are classified as the ten Flat Riemann manifolds \cite{w}. 
These manifolds are described as Seifert manifolds in \cite{orlik},  subsection 8.2  where six are
orientable $M_1,M_2,M_3,M_4, M_5,M_6$ and four non-orientable $N_1,N_2,N_2,N_3,N_4$. \\

The Seifert invariants of these manifolds are given as follows:

$$M_1=\{0;(o_1,1);\},\ M_2=\{-2;(o_1,0);(2,1),(2,1),(2,1),(2,1)\}=\{0;(n_2,2);\},$$
$$M_3=\{-1;(o_1,0);(3,1),(3,1),(3,1)\},\ M_4=\{-1;(o_1,0);(2,1),(4,1),(4,1)\},$$
$$M_5=\{-1;(o_1,0);(2,1),(3,1),(6,1)\},\ M_6=\{-1;(n_2,1);(2,1),(2,1)\},$$
$$N_1=\{0;(n_1,2);\}=\{0;(o_2,1);\},\ N_2=\{1;(n_1,2);\}=\{1;(o_2,1);\},$$
$$N_3=\{0;(n_3,2);\},\ N_4=\{1;(n_3,2);\}=\{0;(n_1,1);(2,1),(2,1)\}.$$

  Now we state the main result  in terms  of involutions. By a free  involution we mean an equivalent classes of free involutions.

\begin{theo}\label{them:flat}  {\bf A)}  $M_1$ admits four  free involutions  $\tau_1, \tau_2, \tau_3, \tau_4$.

\begin{enumerate}

\item The quotient  $M_1/\tau_1$ is homeomorphic  to $M_1$. The $\Z_2$-index of $((M_1,\tau_1);M_1)$ equals 1. 

\item  The quotient  $M_1/\tau_2$ is homeomorphic  to $M_2$. The $\Z_2$-index of $(M_1,\tau_2);M_2)$ equals 1. 

\item  The quotient  $M_1/\tau_3$ is homeomorphic  to $N_1$. The $\Z_2$-index of $((M_1,\tau_3);N_1)$ equals 1. 

\item  The quotient  $M_1/\tau_4$ is homeomorphic  to $N_2$. The $\Z_2$-index of $((M_1,\tau_4);N_2)$ equals 1. 
\end{enumerate}

\bigskip

  {\bf B)}  $M_2$ admits five  free involutions  $\tau_1, \tau_2, \tau_3, \tau_4, \tau_5$.

\begin{enumerate}

\item The quotient  $M_2/\tau_1$ is homeomorphic  to $M_2$. The $\Z_2$-index of $((M_2,\tau_1);M_2)$ equals 2. 

\item  The quotient  $M_2/\tau_2$ is homeomorphic  to $M_4$. The $\Z_2$-index of $((M_2,\tau_2);M_4)$ equals 1. 

\item  The quotient  $M_2/\tau_3$ is homeomorphic  to $N_3$. The $\Z_2$-index of $((M_2,\tau_3);N_3)$ equals 2.

\item  The quotient  $M_2/\tau_4$ is homeomorphic  to $N_4$. The $\Z_2$-index of $((M_2,\tau_4);N_4)$ equals 2.

\item  The quotient  $M_2/\tau_5$ is homeomorphic  to $M_6$. The $\Z_2$-index of $((M_2,\tau_5);M_6)$ equals 2. 

\end{enumerate}

\bigskip

 {\bf C)}  $M_3$ admits  two free involutions  $\tau_1, \tau_2$.

\begin{enumerate}

\item The quotient  $M_3/\tau_1$ is homeomorphic  to $M_3$. The $\Z_2$-index of $((M_3,\tau_1);M_3)$ equals 1. 

\item  The quotient  $M_3/\tau_2$ is homeomorphic  to $M_5$. The $\Z_2$-index of $((M_3,\tau_2);M_5)$ equals 1. 

\end{enumerate}

\bigskip

 {\bf D)} For $M_4$ we have one  free involutions $\tau$.  
The quotient  $M_4/\tau$ is homeomorphic  to $M_4$. The $\Z_2$-index of $((M_4,\tau);M_4)$ equals 3. 

\bigskip

  {\bf E)}  For $M_5$ and $M_6$  there is no involution.
 
 \bigskip

{\bf F)}   For $N_1$ we have seven  free involutions  $\tau_1, \tau_2, \tau_3, \tau_4, \tau_5, \tau_6, \tau_7$. 

\begin{enumerate}

\item The quotient  $N_1/\tau_1$ is homeomorphic  to $N_1$. The $\Z_2$-index of $((N_1,\tau_1);N_1)$ equals 2. 

\item  The quotient  $N_1/\tau_2$ is homeomorphic  to $N_1$. The $\Z_2$-index of $((N_1,\tau_2);N_1)$ equals 1. 

\item  The quotient  $N_1/\tau_3$ is homeomorphic  to $N_2$. The $\Z_2$-index of $((N_1,\tau_3);N_2)$ equals 1.  

\item  The quotient  $N_1/\tau_4$ is homeomorphic  to $N_3$. The $\Z_2$-index of $((N_1,\tau_4);N_3)$ equals 1.

\item  The quotient  $N_1/\tau_5$ is homeomorphic  to $N_3$. The $\Z_2$-index of $((N_1,\tau_5);N_3)$ equals 2.

\item  The quotient  $N_1/\tau_6$ is homeomorphic  to $N_4$. The $\Z_2$-index of $((N_1,\tau_6);N_4)$ equals 1.  

\item  The quotient  $N_1/\tau_7$ is homeomorphic  to $N_4$. The $\Z_2$-index of $((N_1,\tau_7);N_4)$ equals 2. 

\end{enumerate}

\bigskip

 {\bf G)}   $N_2$ admits one  free involution  $\tau$. 
The quotient  $N_2/\tau$ is homeomorphic  to $N_1$. The $\Z_2$-index of $((N_2,\tau);N_1)$ equals 3. 

\bigskip

 {\bf H)}  $N_3$ admits one  free involution  $\tau$. 
The quotient  $N_3/\tau$ is homeomorphic  to $N_3$. The $\Z_2$-index of $((N_3,\tau);N_3)$ equals 3. 

\bigskip

 {\bf I)}  $N_4$ admits one  free involution  $\tau$.  
The quotient  $N_4/\tau$ is homeomorphic  to $N_3$. The $\Z_2$-index of $((N_4,\tau);N_4)$ equals 2.  

 \end{theo}

\subsection{Double coverings and $\Z_2$-indices for Flat manifolds}

 The purpose of this subsection
is to get the graph FIGURE \ref{fig:flat} below where the arrows start from the covering to the base and the number associated with each arrow is the $\Z_2$-index of the covering. This is obtained by proving Propositions 
 from \ref{lem:M1} to \ref{lem:N4}.  This graph is a summary of Theorem 
\ref{them:flat}.

 To state the following propositions, we need to recognize when two double coverings are equivalent in the sense of Definition    \ref{defi:equiv0}. Here our manifolds are the ten Flat manifolds. They verify the Second Bieberbach Theorem \cite{BB1}, \cite {BB2}, \cite{Char}, also \cite{HW}. Hence any automorphism of the fundamental group of these manifolds is realizable as a homeomorphism of the manifold. Therefore we will use the item b) of Lemma \ref{lem:equi}  in the proof of Propositions 5  to 14 to decide when free involutions are equivalent.\\
In the Theorem \ref{them:flat}, to simplify the notations, involutions $\mu_i$  obtained in the following propositions have been reordered and renamed $\tau_j$.  

\begin{figure}[!h]
   \includegraphics[width=9cm]{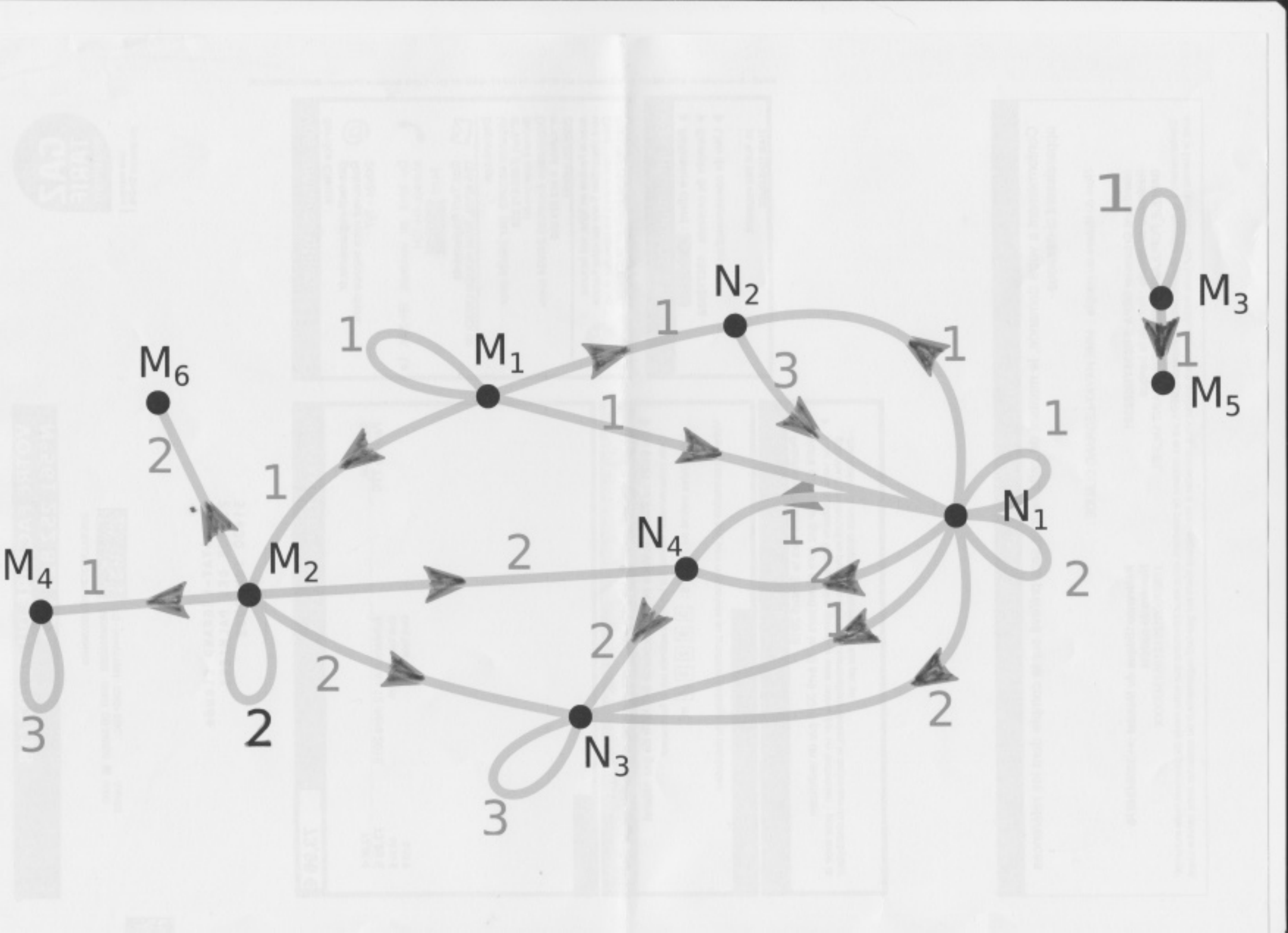}
   \caption{\label{fig:flat} Graph for the family of Flat Seifert manifolds}
  \end{figure}

 We consider the following presentation of $M_1=\{0;(o_1,1);\}$
\begin{eqnarray}
\pi_1(M_1)=\left<{\begin{matrix}v_1,v_2\\h\end{matrix}\left|
\begin{matrix}
v_1hv_1^{-1}h^{-1},&v_2hv_2^{-1}h^{-1}\\
[v_1,v_2]&\end{matrix}\right.}\right>.
\end{eqnarray}

\begin{prop}\label{lem:M1} 
On $M_1$ we have seven  epimorphisms $\varphi: \pi_1(M_1) \to \Z_2$   given by:
$$\begin{array}{lll}
 \varphi_1(v_1)=  1,& \varphi_1(v_2)= 0,&\varphi_1(h)=0\\  
\varphi_2(v_1)= 0,&\varphi_2(v_2)=1,&\varphi_2(h)=0\\  
\varphi_3(v_1)= 1,&  \varphi_3(v_2)=1,&\varphi_3(h)=0\\  
\varphi_4(v_1)= 0,&\varphi_4(v_2)=0,&\varphi_4(h)=1\\ 
\varphi_5(v_1)=1,&\varphi_5(v_2)=0,&\varphi_5(h)=1\\  
\varphi_6(v_1)=0,&\varphi_6(v_2)=1,&\varphi_6(h)=1\\
\varphi_7(v_1)=1,&\varphi_7(v_2)=1,&\varphi_7(h)=1.  
\end{array}$$

 The seven epimorphisms are equivalent.  The associated double covering and free involution form  the pair $(M_1,\mu_1)$.
 
The $\Z_2$-index of $((M_1,\mu_1);M_1)$ equals 1.
 \end{prop}
 
 \begin{proof}
Since $\pi_1(M_1)=\Z+\Z+\Z$, using the automorphisms which change the variables, we obtain that $\varphi_1, \varphi_2$ and $\varphi_4$ are equivalent and also  $\varphi_3,\varphi_5$ and $\varphi_6$ are equivalent. Using the automorphism defined on the generators by $v_1\mapsto v_1v_2, v_2\mapsto v_2, h\mapsto h$, we obtain that $\varphi_7$ is equivalent to $\varphi_6$, and $\varphi_3$ is equivalent to $\varphi_2$.

 \bigskip
 
 Reidemeister-Schreier algorithm gives $M_1$ as the total space of the double covering of $M_1$.
 
 \bigskip
 
 Since $\pi_1(M_1)=\Z+\Z+\Z$ is  free abelian,  it is straightforward to see that  $\varphi_1$   factors 
through $\Z \to \Z_2$, hence the $\Z_2$-index of $((M_1,\mu_1);M_1)$ equals 1. \\
\end{proof}

\bigskip

We consider the following presentation of $M_2=\{-2;(o_1,0);(2,1),(2,1),(2,1),(2,1)\}$
\begin{eqnarray}
\pi_1(M_2)=\left<{\begin{matrix}s_1,s_2,s_3, s_4\\h\end{matrix}\left|
\begin{matrix}
 [s_k,h],1\le k\le 4;& s_1^2h, s_2^2h ,s_3^2h, s_4^2h\\
s_1s_2s_3s_4=h^{-2}&\end{matrix}\right.}\right>. 
\end{eqnarray} 

\begin{prop}\label{lem:M2} 
On $M_2$ we have seven  epimorphisms $\varphi: \pi_1(M_2) \to \Z_2$   given by: 
 $$\begin{array}{lllll}
\varphi_1(s_1)=  1,&\varphi_1(s_2)= 1,&\varphi_1(s_3)=0,&\varphi_1(s_4)=0, &\varphi_1(h)=0\\
\varphi_2(s_1)=1,&\varphi_2(s_2)=0,&\varphi_2(s_3)=1,&\varphi_2(s_4)=0,&
\varphi_2(h)=0\\
\varphi_3(s_1)= 1,&\varphi_3(s_2)=0,&\varphi_3(s_3)=0,&\varphi_3(s_4)=1,& \varphi_3(h)=0\\
\varphi_4(s_1)=0,&\varphi_4(s_2)=1,&\varphi_4(s_3)=1,&\varphi_4(s_4)=0,&
\varphi_4(h)=0\\ 
\varphi_5(s_1)=0,&\varphi_5(s_2)=1,&\varphi_5(s_3)=0,&\varphi_5(s_4)=1,&
\varphi_5(h)=0\\ 
\varphi_6(s_1)=0,&\varphi_6(s_2)=0,&\varphi_6(s_3)=1,&\varphi_6(s_4)= 1,& \varphi_6(h)=0\\
\varphi_7(s_1)=1,&\varphi_7(s_2)=1,&\varphi_7(s_3)=1,&\varphi_7(s_4)=1,&
\varphi_7(h)=0. 
 \end{array}$$

The six  first $\varphi's$ are equivalent. The double covering and free involution which correspond to $\varphi _1$ form the pair $(M_2,\mu_1)$. The $\Z_2$-index of $((M_2,\mu_1);M_2)$ equals 2.

  The double covering and free involution which correspond to $\varphi _7$ form the pair $(M_1,\mu_2)$. The $\Z_2$-index of $((M_1,\mu_2);M_2)$ equals 1.

 \end{prop}
 
 \begin{proof}
 Let us define an automorphism $\theta$  by $\theta(s_i)=s_is_{i+1}s_i^{-1},$  $\theta(s_{i+1})=s_i$, and the identity on the other variables. Its effect is to exchange the generators $s_i$ and $s_{i+1}$.  We obtain  the equivalence of the six  first $\varphi's$.  

 Reidemeister-Schreier algorithm gives $M_2$ as the total space of the double covering of $M_2$ determined by $\varphi_1$.

 Because of the relation $s_3^2h=1$, any lift of  $\varphi_1$ must send $h$ to an even number in $\Z$. The long relation implies that $h$ goes to 0 and this lift is trivial. The epimorphism $\varphi_1$  does not    factor through $\Z \to \Z_2$. 
 
 We are in Case 3 of  Proposition (\ref{prop:cube}), hence $[\varphi_1]^3=0$ and 
the $\Z_2$-index of $((M_2,\mu_1);M_2)$ equals 2.

\bigskip

  Reidemeister-Schreier algorithm gives $M_1$ as the total space of the double covering of $M_2$ determined by $\varphi_7$.

Define a factorization of   $\varphi_7$ by $\psi$ such that
$\psi(s_1)=  1,\psi(s_2)= 1,\psi(s_3)=1,\psi(s_4)=1,\psi (h)=-2$. The $\Z_2$-index of $((M_1,\mu_2);M_2)$ equals 1.

\end{proof}

\bigskip

We consider the following presentation of $M_3=\{-1;(o_1,0);(3,1),(3,1),(3,1)\}$
\begin{eqnarray}
\pi_1(M_3)=\left<{\begin{matrix}s_1,s_2,s_3\\h\end{matrix}\left|
\begin{matrix}
 [s_k,h],1\le k\le 3;& s_1^3h, s_2^3h ,s_3^3h\\
s_1s_2s_3=h^{-1}&\end{matrix}\right.}\right>. 
\end{eqnarray} 

\begin{prop}\label{lem:M3} 
On $M_3$ we have one   epimorphisms $\varphi: \pi_1(M_3) \to \Z_2$    given by:
$$\begin{array}{l}
\varphi(s_1)=1,\varphi(s_2)=1,\varphi(s_3)= 1,\varphi(h)=1
\end{array}$$

The   double covering and free involution form the pair $(M_3,\mu)$. The $\Z_2$-index of $((M_3,\mu);M_3)$ equals 1.
\end{prop}

\begin{proof}

Using Reidemeister-Shreier algorithm, we find 
$\{-2;(o_1,0);(3,2),(3,2),(3,2)\}$  as Seifert invariants of  the double covering determined by $\varphi$. After a change of orientation [Seifert],   we recognize those of $M_3$. Hence $M_3$ is the total space of the double covering of $M_3$ determined by $\varphi$.

 Define the factorization of $\varphi$ by  $\psi(h)=-3$ and $\psi(s_1)=\psi (s_2)=\psi(s_3)=1$. The $\Z_2$-index of $((M_3,\mu);M_3)$ equals 1.
  
\end{proof}

\bigskip

We consider the following presentation of $M_4=\{-1;(o_1,0);(2,1),(4,1),(4,1)\}$
\begin{eqnarray}
\pi_1(M_4)=\left<{\begin{matrix}s_1,s_2,s_3\\h\end{matrix}\left|
\begin{matrix}
 [s_k,h],1\le k\le 3;& s_1^2h, s_2^4h ,s_3^4h\\
s_1s_2s_3=h^{-1}&\end{matrix}\right.}\right>. 
\end{eqnarray} 

\begin{prop}\label{lem:M4} 
On $M_4$ we have  three  epimorphisms $\varphi: \pi_1(M_4) \to \Z_2$    given by:
$$\begin{array}{llll}
\varphi_1(s_1)= 1,&\varphi_1(s_2)=1,&\varphi_1(s_3)=0,&\varphi_1(h)= 0\\ 
\varphi_2(s_1)=1,&\varphi_2(s_2)=0,&\varphi_2(s_3)=1,&\varphi_2(h)=0\\ 
\varphi_3(s_1)= 0,&\varphi_3(s_2)=1,&\varphi_3(s_3)=1,&\varphi_3(h)= 0.
\end{array}$$
\begin{enumerate}
\item 
The epimorphisms $\varphi_1, \varphi_2$ are equivalent.  The double covering and free involution which correspond to $\varphi _1$ form the pair $(M_4,\mu)$. The $\Z_2$-index of $((M_4,\mu);M_4)$ equals 3.
\item
 The double covering and free involution corresponding to $\varphi_3$ form  the pair $(M_2,\mu_2)$. The $\Z_2$-index of $((M_2,\mu_2);M_4)$ equals 1.
\end{enumerate}

 \end{prop}
 \begin{proof}
 \begin{enumerate}
\item The generators $s_2$ and $s_3$ are exchanged using the same automorphism $\theta$ of $\pi_1(M_4)$ defined by $\theta(s_1)=s_1,\theta(s_2)=s_2s_3s_2^{-1}, \theta(s_3)=s_2, \theta(h)=h$.    We obtain that $\varphi_1, \varphi_2$ are equivalent. 
 
 Reidemeister-Schreier algorithm gives $M_4$ as the total space of the double covering of $M_4$ determined by $\varphi_1$.

 We are in Case 3 of Proposition  (\ref{prop:cube}) and $[\varphi_2]^3\neq 0$. The $\Z_2$-index of $((M_4,\mu);M_4)$ equals 3.
 
 \bigskip
 
\item    Reidemeister-Schreier algorithm gives $M_2$ as the total space of the double covering of $M_4$ determined by $\varphi_3$.

For $\varphi_3$ define a lift by $\psi(s_1)= 2,\psi(s_2)=1,\psi(s_3)=1,\psi(h)= -4$. The $\Z_2$-index of $((M_2,\mu_2);M_4)$ equals 1. 

\end{enumerate}
\end{proof}

\bigskip

We consider the following presentation of $M_5=\{-1;(o_1,0);(2,1),(3,1),(6,1)\}$
\begin{eqnarray}
\pi_1(M_5)=\left<{\begin{matrix}s_1,s_2,s_3\\h\end{matrix}\left|
\begin{matrix}
 [s_k,h],1\le k\le 3;& s_1^2h, s_2^3h ,s_3^6h\\
s_1s_2s_3=h^{-1}&\end{matrix}\right.}\right>. 
\end{eqnarray} 

\begin{prop}\label{lem:M5} 
On $M_5$ we have  one epimorphisms $\varphi: \pi_1(M_5) \to \Z_2$    given by:
$$\begin{array}{llll}
\varphi(s_1)=1, &\varphi(s_2)=0,& \varphi(s_3)=1,&\varphi(h)=0.
\end{array}$$

The   double covering and free involution form the pair $(M_3,\mu)$. The $\Z_2$-index of $((M_3,\mu);M_5)$ equals 1.
\end{prop}

\begin{proof}

 Reidemeister-Schreier algorithm gives $M_3$ as the total space of the double covering of $M_5$ determined by $\varphi$.

Define a factorization of $\varphi $ by $\psi(h)=-6$ and $\psi(s_1)=3, \psi(s_2)=2, \psi(s_3)=1$. The $\Z_2$-index of $((M_3,\mu);M_3)$ equals 1.

\end{proof}

\bigskip

We consider the following presentation of  $ M_6=\{-1;(n_2,1);(2,1),(2,1)\}$
\begin{eqnarray}
\pi_1(M_6)=\left<{\begin{matrix}s_1,s_2 \\v\\h\end{matrix}\left|
\begin{matrix}
 [s_k,h], k=\{1, 2\} ;& s_1^2h, s_2^2h\\
vhv^{-1}h&\\
s_1 s_2v^2h&\end{matrix}\right.}\right>.
\end{eqnarray}

\begin{prop}\label{lem:M6} 
 On $M_6$ we have  three  epimorphisms $\varphi: \pi_1(M_6) \to \Z_2$    given by:
$$\begin{array}{llll}
\varphi_1(s_1)=1,&\varphi_1(s_2)= 1,&\varphi_1(v)=0,&\varphi_1(h)=0\\
\varphi_2(s_1)=1,&\varphi_2(s_2)=1,&\varphi_2(v)=1,&\varphi_2(h)=0 \\
\varphi_3(s_1)=0,&\varphi_3(s_2)=0,&\varphi_3(v)=1,&\varphi_3(h)=0.
\end{array}$$

The three epimorphisms are equivalent. The double covering and free involution form the pair  $(M_2,\mu_5)$.  The $\Z_2$-index of  $((M_2,\mu_5);M_6)$ equals 2.

\end{prop}

\begin{proof}
 Define the morphism $\theta_1$ of $\pi_1(M_6)$ by $\theta_1(v)=v'$ with $v'= s_2v$,  $\theta_1(s_2)= v's_2^{-1}v'^{-1}$ and the identity on the others variable. It is easy to prove that $\theta_1$ is an automorphism and $\varphi_2\circ\theta_1=\varphi_1.$ Hence $\varphi_1$ and $\varphi_2$ are equivalent.
 
Using the following equalities of presentation of $\pi_1(M_6)$, we will explicit an automorphism $\theta$ of $\pi_1(M_6)$ such that $\varphi_1=\varphi_3\circ\theta$.

$\pi_1(M_6)=<s_1,s_2,v,h\tq [s_1,h]=[s_2,h]=1, s_1^2=s_2^2=h^{-1},vhv^{-1}=h^{-1}, s_1s_2v^2=h^{-1}>=
<s_1,s_2,v\tq s_1^2=s_2^2=vs_1^{-2}v^{-1}, s_2v^2=s_1>=
<s_1,v\tq s_1^2=(s_1v^{-2})^2=vs_1^{-2}v^{-1}>=<s_1,v\tq s_1v^2s_1^{-1}=v^{-2}, vs_1^2v^{-1}=s_1^{-2}>$.

Let us define $\theta$ by $\theta(s_1)=v, \theta(s_2)=vs_1^{-2}, \theta(v)=s_1, \theta(h)=v^{-2}$ and $\theta^{-1}(s_1)=v, \theta^{-1}(s_2)=s_1v^{-2}, \theta^{-1}(v)=s_1, \theta^{-1}(h)=s_1^{-2}.$ We verify that $\theta^{-1}$ is a morphism  inverse of the morphism $\theta$ and $\varphi_1=\varphi_3\circ\theta$. Hence the three epimorphisms are equivalent. 

 Reidemeister-Schreier algorithm gives $M_2=\{0;(n_2,0); \}$ as the total space of the double covering of $M_6$ determined by $\varphi_1$.

The epimorphism $\varphi_1$ does not admit a lift from $\pi_1(M_6)$ to $\Z$. We are in Case 3 of Proposition (\ref{prop:cube}) and $[\varphi_1]^3=0$. The $\Z_2$-index of  $((M_2,\mu_5);M_6)$ equals 2.

\end{proof}

\bigskip

We consider the following presentation of $N_1=\{0;(n_1,2);\}$:
\begin{eqnarray}
\pi_1(N_1)=\left<{\begin{matrix}v_1  ,v_2\\h\end{matrix}\left|
\begin{matrix}
 v_jhv_j^{-1}h^{-1},j=1,2&\\
v_1^2v_2^2&\end{matrix}\right.}\right>. 
\end{eqnarray} 

\begin{prop}\label{lem:N1} 
 On $N_1$ we have  seven  epimorphisms $\varphi: \pi_1(N_1) \to \Z_2$    given by:
$$\begin{array}{lll}
\varphi_1(v_1)= 1,&\varphi_1(v_2)=1,&\varphi_1(h)=1\\ 
\varphi_2(v_1)=1,&\varphi_2(v_2)=0,&\varphi_2(h)=1\\  
\varphi_3(v_1)=0,&\varphi_3(v_2)=1,&\varphi_3(h)=1\\  
\varphi_4(v_1)=0,&\varphi_4(v_2)=0,&\varphi_4(h)=1\\ 
\varphi_5(v_1)=1,&\varphi_5(v_2)=1,&\varphi_5(h)=0\\
\varphi_6(v_1)=1,&\varphi_6(v_2)=0,&\varphi_6(h)=0\\ 
\varphi_7(v_1)=0,&\varphi_7(v_2)=1,&\varphi_7(h)=0.
\end{array}$$ 
\begin{enumerate}
\item 
The epimorphisms $\varphi_1$ and $ \varphi_4$ are equivalent. The corresponding  double covering and free involution form the pair $(N_1, \mu _4)$.  The $\Z_2$-index of $((N_1,\mu_4);N_1)$ equals 1.
\item
 The epimorphisms $\varphi_2$ and $\varphi_3$ are equivalent. The corresponding  double covering and free involution form the pair $(N_2,\mu_2) $. The $\Z_2$-index of $((N_2,\mu_2);N_1)$ equals 3.
\item
 The epimorphism $\varphi_5$ has $(M_1,\mu_5)$  as corresponding double covering and free involution.   The $\Z_2$-index of $((M_1,\mu_5);N_1)$ equals 1.
\item
 The epimorphisms $\varphi_6$ and $ \varphi_7$ are equivalent. The corresponding  double covering and free involution are the pair $(N_1,\mu_6 )$. The $\Z_2$-index of $((N_1,\mu_6);N_1)$ equals 2.
\end{enumerate}
\end{prop}

\begin{proof}
\begin{enumerate}
 \item Let us define an automorphism $\theta$ of $\pi_1(N_1)$ by $\theta(v_1)=hv_1$, $\theta(v_2)=h^{-1}v_2$ and $\theta(h)=h$. It verifies $\varphi_1\circ\theta=\varphi_4$. Hence $\varphi_1$ and $\varphi_4$ are equivalent.

 Reidemeister-Schreier algorithm gives $N_1$ as the total space of the double covering of $N_1$ determined by $\varphi_1$.

The morphism $\psi$ defined  by $\psi(v_1)=2, \psi(v_2)=-2 ,\psi(h)=1$ is a lift of $\varphi_1$. This implies that the $\Z_2$-index of $((N_1,\mu_4);N_1)$ equals 1.

\item Let us define an automorphism $\theta$ of $\pi_1(N_1)$ by $\theta(v_1)=v_1v_2v_1^{-1}$, $\theta(v_2)=v_1$ and $\theta(h)=h$. It verifies $\varphi_2\circ\theta=\varphi_3$. Hence $\varphi_2$ and $\varphi_3$ are equivalent.

 Reidemeister-Schreier algorithm gives $N_2$ as the total space of the double covering of $N_1$ determined by $\varphi_2$.

We are in Case 1, 2) of Proposition (\ref{prop:cube}) with $c=0$ and only one $\varphi_2(v_1)\neq 0$ hence $[\varphi_2]^3\neq 0 $ and the $\Z_2$-index of $((N_2,\mu_2);N_1)$ equals 3.

\item Reidemeister-Schreier algorithm gives $M_1$ as the total space of the double covering of $N_1$ determined by $\varphi_5$.

The morphism $\psi$ defined  by $\psi(v_1)=1, \psi(v_2)=-2,\psi(h)=0$ is a lift of $\varphi_5$. The $\Z_2$-index of $((M_1,\mu_5);N_1)$ equals 1.

\item   As above let us define an automorphism $\theta$ of $\pi_1(N_1)$ by $\theta(v_1)=v_1v_2v_1^{-1}$, $\theta(v_2)=v_1$ and $\theta(h)=h$. It verifies $\varphi_6\circ\theta=\varphi_7$. Hence $\varphi_6$ and $\varphi_7$ are equivalent.
Reidemeister-Schreier algorithm gives $N_1$ as the total space of the double covering of $N_1$ determined by $\varphi_6$.

The epimorphism does not admit a lift to $\Z$ and
we are in Case 1, 2) of Proposition (\ref{prop:cube}) but now $\varphi_6(h)=0$. It means that $[\varphi_6]^3=0$ and the $\Z_2$-index of $((N_1,\mu_6);N_1)$ equals 2.
\end{enumerate}
\end{proof}

\bigskip

 We consider the following presentation of $N_2=\{1;(n_1,2);\}$:
\begin{eqnarray}
\pi_1(N_2)=\left<{\begin{matrix}v_1  ,v_2\\h\end{matrix}\left|
\begin{matrix}
 v_jhv_j^{-1}h^{-1},j=1,2&\\
v_1^2v_2^2h^{-1}&\end{matrix}\right.}\right>. 
\end{eqnarray} 

\begin{prop}\label{lem:N2} 
 On $N_2$ we have  three  epimorphisms $\varphi: \pi_1(N_2) \to \Z_2$    given by:
$$\begin{array}{lll}
\varphi_1(v_1)=1,&\varphi_1(v_2)=1,&\varphi_1(h)=0\\ 
\varphi_2(v_1)=1,&\varphi_2(v_2)=0,&\varphi_2(h)=0\\
\varphi_3(v _1)=0,&\varphi_3(v_2)=1,&\varphi_3(h)=0.
\end{array}$$

\begin{enumerate}
\item

 The epimorphism $\varphi_1$ has $(M_1,\mu _1)$  as corresponding double covering and free involution.   The $\Z_2$-index of $((M_1,\mu_1);N_2)$ equals 1.

\item  The epimorphisms $\varphi_2$ and $\varphi_3$ are equivalent.  They have $(N_1,\mu_2) $  as corresponding double covering and free involution.  The $\Z_2$-index of $((N_1,\mu_2);N_2)$ equals 1.
\end{enumerate}
\end{prop}

\begin{proof}
\begin{enumerate}
\item Reidemeister-Schreier algorithm gives $M_1$ as the total space of the double covering of $N_2$ determined by $\varphi_1$.

The epimorphism $\varphi_1$ admits a lift to $\Z$ defined by $\psi(v_1)=1$ and $\psi(v_2)=\psi(h)=0$. The $\Z_2$-index of $((M_1,\mu_1);N_2)$ equals 1.

\item As in the previous proposition, let us define an automorphism $\theta$ of $\pi_1(N_2)$ by $\theta(v_1)=v_1v_2v_1^{-1}$, $\theta(v_2)=v_1$ and $\theta(h)=h$. It verifies $\varphi_2\circ\theta=\varphi_3$. Hence $\varphi_2$ and $\varphi_3$ are equivalent.

 Reidemeister-Schreier algorithm gives $N_1$ as the total space of the double covering of $N_2$ determined by $\varphi_2$.

The epimorphism $\varphi_2$ admits a lift to $\Z$ defined by $\psi(v_1)=1,\psi(v_2)=0$ and $\psi(h)=-2$. The $\Z_2$-index of $((N_1,\mu_2);N_2)$ equals 1.\\
\end{enumerate}
\end{proof}

\bigskip

We consider the following presentation of $N_3=\{0;(n_3,2);\}$:
\begin{eqnarray}
\pi_1(N_3)=\left<{\begin{matrix}v_1  ,v_2\\h\end{matrix}\left|
\begin{matrix}
 v_1hv_1^{-1}h^{-1}, v_2hv_2^{-1}h&\\
v_1^2v_2^2&\end{matrix}\right.}\right>. 
\end{eqnarray} 

\begin{prop}\label{lem:N3} 
 On $N_3$ we have  seven  epimorphisms $\varphi: \pi_1(N_3) \to \Z_2$   given by:
$$\begin{array}{lll}
\varphi_1(v_1)=1,&\varphi_1(v_2)=1,&\varphi_1(h)=1\\
\varphi_2(v_1)=1,&\varphi_2(v_2)=0,&\varphi_2(h)=1\\  
\varphi_3(v_1)=0,&\varphi_3(v_2)=1,&\varphi_3(h)=1\\
\varphi_4(v_1)=0,&\varphi_4(v_2)=0,&\varphi_4(h)=1\\  
\varphi_5(v_1)=1,&\varphi_5(v_2)=1,&\varphi_5(h)=0\\ 
\varphi_6(v_1)=1,&\varphi_6(v_2)= 0,&\varphi_6(h)=0\\ 
\varphi_7(v_1)=0,&\varphi_7(v_2)=1,&\varphi_7(h)=0.
\end {array}$$
\begin{enumerate}
\item The epimorphisms $\varphi_1$ and $\varphi_2$ are equivalent.  They have $(N_4,\mu_1) $  as corresponding double covering and free involution.  The $\Z_2$-index of $((N_4,\mu_1);N_3)$ equals 2.

\item The epimorphisms $\varphi_3$ and $\varphi_4$ are equivalent.  They have $(N_3,\mu_3) $  as corresponding double covering and free involution.  The $\Z_2$-index of $((N_3,\mu_3);N_3)$ equals 3.

\item The epimorphism $\varphi_5$ has $(N_1,\mu_5) $  as corresponding double covering and free involution.   The $\Z_2$-index of $((N_1,\mu_5);N_3)$ equals 1.

\item The epimorphism $\varphi_6$ has $(N_3,\mu_6) $  as corresponding double covering and free involution.    The $\Z_2$-index of $(M_2\mu_6);N_3$ equals 2.
 
\item The epimorphism $\varphi_7$ has $(N_1,\mu_7) $  as corresponding  double covering and free involution.   The $\Z_2$-index of $((N_1,\mu_7);N_3)$ equals 2. 
\end{enumerate}
\end{prop}

\begin{proof}
\begin{enumerate}
\item The bijection  $\theta(v_1)=v_1,\theta(v_2)=v_2h,\theta(h)=h$ defines an automorphism  $\theta$ of $\pi_1(N_2)$ such that $\varphi_1\circ\theta=\varphi_2$.  Hence the epimorphisms $\varphi_2$ and $\varphi_1$ are equivalent.

 Reidemeister-Schreier algorithm gives $N_4$ as the total space of the double covering of $N_3$ determined by $\varphi_1$.

The epimorphism $\varphi_1$ does not admit a lift to $\Z$. We are in Case 1 4) of Proposition (\ref{prop:cube}) with $c=0, g=2, \varphi_1(v_1)=1$ which implies $[\varphi_1]^3=0.$ The $\Z_2$-index of $((N_4,\mu_1);N_3)$ equals 2.

\item The same automorphism  $\theta$ of $\pi_1(N_2)$ verifies  $\varphi_3\circ\theta=\varphi_4$.  Hence the epimorphisms $\varphi_3$ and $\varphi_4$ are equivalent.

 Reidemeister-Schreier algorithm gives $N_3$ as the total space of the double covering of $N_3$ determined by $\varphi_3$.

We are in Case 1, 4) of Proposition (\ref{prop:cube}) with $c=0, g=2$ but now $\varphi_3(v_1)=0$. Hence the cup-cube $[\varphi_3]^3\neq 0$ and the $\Z_2$-index of $((N_3,\mu_3);N_3)$ equals 3.

\item    Reidemeister-Schreier algorithm gives $N_1$ as the total space of the double covering of $N_3$ determined by $\varphi_5$.

The epimorphism $\varphi_5$ admits a lift to $\Z$ defined by $\psi(v_1)=\psi(v_2)=1$ and $\psi(h)=0$. The $\Z_2$-index of $((N_1,\mu_5);N_3)$ equals 1.

\item   Reidemeister-Schreier algorithm gives $M_2$ as the total space of the double covering of $N_3$ determined by $\varphi_6$.

The epimorphism $\varphi_6$ does not admit a lift and $\varphi_6(h)=0$. Hence $[\varphi_6]^3=0$ and  the $\Z_2$-index of $((M_2,\mu_6);N_3)$ equals 2.

(\item Using Reidemeister-Shreier algorithm, we find $N_1$ is the total space of the double covering of $N_3$ determined by $\varphi_7$.

The epimorphism $\varphi_7$ does not admit a lift and $\varphi_7(h)=0$. Hence $[\varphi_7]^3=0$ and the $\Z_2$-index of $((N_1,\mu_7);N_3)$ equals 2.\\
\end{enumerate}
\end{proof}

\bigskip

We consider the following presentation of $N_4=\{1;(n_3,2);\}$:
\begin{eqnarray}
\pi_1(N_4)=\left<{\begin{matrix}v_1  ,v_2\\h\end{matrix}\left|
\begin{matrix}
 v_1hv_1^{-1}h^{-1}, v_2hv_2^{-1}h&\\
v_1^2v_2^2h^{-1}&\end{matrix}\right.}\right>. 
\end{eqnarray} 

\begin{prop}\label{lem:N4} 
 On $N_4$ we have  three  epimorphisms   $\varphi: \pi_1(N_4) \to \Z_2$  given by:
$$\begin{array}{lll}
\varphi_1(v_1)=1,&\varphi_1(v_2)=1,&\varphi_1(h)=0\\ 
\varphi_2(v_1)=1,&\varphi_2(v_2)=0,&\varphi_2(h)=0\\ 
\varphi_3(v_1)=0,&\varphi_3(v_2)=1,&\varphi_3(h)=0.
\end{array}$$
\begin{enumerate}
\item The epimorphism $\varphi_1$ gives as corresponding double covering and free involution the pair $(N_1,\mu_1)$. The $\Z_2$-index of $((N_1,\mu_1);N_4)$ equals 1.

\item The epimorphism $\varphi_2$ gives as corresponding double covering and free involution the pair $(M_2,\mu_2)$. The $\Z_2$-index of $((M_2,\mu_2);N_4)$ equals 2.

\item The epimorphism $\varphi_3$ gives as corresponding  double covering and free involution the pair $(N_1,\mu_3)$. The $\Z_2$-index of $((N_1,\mu_3);N_4)$ equals 2. 
\end{enumerate}
\end{prop}

\begin{proof}
\begin{enumerate}
\item   Reidemeister-Schreier algorithm gives $N_1$ as the total space of the double covering of $N_4$ determined by $\varphi_1$.

The epimorphism $\varphi_1$ admits a lift $\psi$ on $\Z$ given by $\psi(v_1)=1,\psi(v_2)=-1,\psi(h)=0$. The $\Z_2$-index of $((N_1,\mu_1);N_4)$ equals 1.

\item  Reidemeister-Schreier algorithm gives $M_2$ as the total space of the double covering of $N_4$ determined by $\varphi_2$.

The epimorphism $\varphi_2$ does not admit a lift to $\Z$. We are in Case 1, 4) of Proposition (\ref{prop:cube}) with $\varphi_2(h)=0$. Hence $[\varphi_2]^3=0$ and the $\Z_2$-index of $((M_2,\mu_2);N_4)$ equals 2.

\item As for $N_3$, Reidemeister-Shreier algorithm gives $N_1$ is the total space of the double covering of $N_4$ determined by $\varphi_3$.

The epimorphism $\varphi_3$ does not admit a lift to $\Z$. We are again in  Case 1, 4) of Proposition (\ref{prop:cube}) with $\varphi_3(h)=0$. Hence $[\varphi_3]^3=0$ and the $\Z_2$-index of $((N_1,\mu_3);N_4)$ equals 2.
\end{enumerate}
\end{proof}

{\it Proof of Theorem \ref{them:flat}:}
\begin{proof} The proof  of part  $A)$- Given a free involution $\tau$ on $M_1$ then the orbit space $M_1/\tau$ is homeomorphic to a manifold $S$ where $S$ is one of the 10 Flat manifolds. Therefore the involution $\tau$ is equivalent to one of the involutions obtained by taking double covering of $S$. So we look at the Propositions from   5 to 14 and look where we have a free  involution on $M_1$. 
Namely, Proposition  \ref{lem:M1}
provides one equivalent classe of involutions on $M_1$ and this classe of involutions satisfies the properties of the item $(1)$ of
part $A)$.  
Proposition  \ref{lem:M2} item $(2)$ 
provides one equivalent classe of involutions on $M_1$ and this classe of involutions satisfies the properties of the item $(2)$ of
part $A)$.
Proposition  \ref{lem:N1} item $(3)$ 
provides one equivalent classe of involutions on $M_1$ and this classe of involutions satisfies the properties of the item $(3)$ of
part $A)$.  
  Proposition  \ref{lem:N2} item $(1)$ 
provides one equivalent classe of involutions on $M_1$ and this classe of involution satisfies the properties of the item $(4)$ of
part $A)$. 

The remain Propositions 7, 8, 9, 10, 13,14 do not provide free involutions on $M_1$ and the result follows. 
The proof of the other parts of the Theorem is completely similar of the proof of part $A)$ and we leave it for 
the reader. 
\end{proof}

\begin{rem} The computation of the $\Z_2$-indices proves that there does not exist an equivariant homeomorphism between the two double coverings   $(N_1,\mu_4)$ and $(N_1,\mu_6)$ of $N_1$ (see Proposition \ref{lem:N1}),
 $(N_1,\mu_5)$ and $(N_1,\mu_7)$ of $N_3$ (see Proposition \ref{lem:N3}), and $(N_1,\mu_1)$ and $(N_1,\mu_3)$ of $N_4$ (see Proposition \ref{lem:N4}).
\end{rem}

\author{Anne Bauval}\\
\address{\small Institut de Math\'ematiques de Toulouse\\
Equipe Emile Picard, UMR 5580\\
Universit\'e Toulouse III\\
118 Route de Narbonne, 31400 Toulouse - France\\
e-mail: bauval@math.univ-toulouse.fr}

\author{Daciberg L.\ Gon\c calves}\\
\address{\small Departamento de Matem\'atica - IME-USP\\
Rua~do~Mat\~ao~1010\\ 
CEP: ~05508-090 - S\~ao Paulo - SP - Brasil\\
e-mail: dlgoncal@ime.usp.br}

\author{Claude Hayat}\\
\address{\small Institut de Math\'ematiques de Toulouse\\
Equipe Emile Picard, UMR 5580\\
Universit\'e Toulouse III\\
118 Route de Narbonne, 31400 Toulouse - France\\
e-mail: hayat@math.univ-toulouse.fr}

\end{document}